\documentclass[reqno]{amsart}
\usepackage{amssymb,graphicx,enumerate,vmargin,hyperref,textcomp}

\setpapersize{A4}
\setmarginsrb{2.7cm}{2.7cm}{2.7cm}{2.7cm}{10pt}{20pt}{10pt}{20pt}
\graphicspath{{./}{figures/}}

\theoremstyle{plain}
\newtheorem{thm}{Theorem}

\newtheorem{lem}[thm]{Lemma}

\theoremstyle{definition}
\newtheorem{remark}[thm]{Remark}
\newtheorem{example}[thm]{Example}
\newtheorem{conv}[thm]{Convention}

\newcommand{\ve}{\varepsilon}
\newcommand{\bve}{{\overline{\ve}}}
\newcommand{\bd}{\overleftarrow}
\newcommand{\fd}{\overrightarrow}
\newcommand{\opti}
{\raisebox{-.04cm}{$\stackrel{{\scriptscriptstyle 0}}{{\scriptscriptstyle 
1}}$}}

\newcommand{\N}{{{\mathbb N}}}
\newcommand{\R}{{{\mathbb R}}}
\newcommand{\Q}{{{\mathbb Q}}}
\newcommand{\Z}{{{\mathbb Z}}}

\newcommand{\bx}{{\mathbf{x}}}

\newcommand{\orb}{{\operatorname{orb}}}
\newcommand{\lhe}{{\operatorname{lhe}}}
\newcommand{\rhe}{{\operatorname{rhe}}}

\newcommand{\hf}{{\widehat{f}}}
\newcommand{\hw}{{\widehat{w}}}
\newcommand{\hI}{{\widehat{I}}}

\begin{document}

\title{Itineraries for inverse limits of tent maps: a backward view}

\date{September 2017}

\author{Philip Boyland}
\address{Department of
    Mathematics\\University of Florida\\372 Little Hall\\Gainesville\\
    FL 32611-8105, USA}
\email{boyland@ufl.edu}

\author{Andr\'e de Carvalho}
\address{Departamento de
    Matem\'atica Aplicada\\ IME-USP\\ Rua Do Mat\~ao 1010\\ Cidade
    Universit\'aria\\ 05508-090 S\~ao Paulo SP\\ Brazil}
\email{andre@ime.usp.br}

\author{Toby Hall}
\address{Department of Mathematical Sciences\\ University of
    Liverpool\\ Liverpool L69 7ZL, UK}
\email{t.hall@liverpool.ac.uk}


\thanks{The authors would like to thank the
  anonymous referee for pointing out several errors in the first draft of this
  paper, and for suggestions for improving the exposition. This work was
  supported by FAPESP [grant number 2016/04687-9]; and by EU Marie-Curie IRSES
  Brazilian-European partnership in Dynamical Systems [grant number
  FP7-PEOPLE-2012-IRSES 318999 BREUDS]. 
  \\\phantom{x}\quad To appear in Topol.\ Appl.\ (2017),
  \url{https://doi.org/10.1016/j.topol.2017.09.012}. 
  \textcopyright\ 2017. This
  manuscript version is made available under the CC-BY-NC-ND 4.0 license
  \url{http://creativecommons.org/licenses/by-nc-nd/4.0/} 
  }

\begin{abstract}
Previously published admissibility conditions for an element of
$\{0,1\}^\Z$ to be the itinerary of a point of the inverse limit of a
tent map are expressed in terms of forward orbits. We give necessary
and sufficient conditions in terms of backward orbits, which is more
natural for inverse limits. These backward admissibility conditions
are not symmetric versions of the forward ones: in particular, the
maximum backward itinerary which can be realised by a tent map mode
locks on intervals of kneading sequences.
\end{abstract}

\maketitle

\section{Introduction}

Inverse limits of tent maps have been much investigated, not only
because of their intrinsic interest as topological spaces, but also
because they are closely related to other topics in dynamical systems
such as hyperbolic attractors and H\'enon maps. A recent highlight is
the proof by Barge, Bruin, and \v Stimac of the Ingram
Conjecture~\cite{Ingram}, which states that the inverse limits of
distinct tent maps are non-homeomorphic.

Kneading theory is widely used in the study of the dynamics of
unimodal maps, and has been extended to and applied in the context of
inverse limits of tent maps by several authors
(e.g.~\cite{BD,Asymp-Bruin}). A key starting point for the application
of such symbolic techniques is understanding the \emph{admissibility
  conditions} under which a sequence of symbols is realised as the
itinerary of a point of the inverse limit. In previous works, such
admissibility conditions have been adapted from those for kneading
theory of unimodal maps, and as such are based on the forward
itineraries of points. This is somewhat unnatural in the context of
inverse limits, where the main focus is on backward orbits.

In this paper we develop admissibility conditions for inverse limits
which are based on backward itineraries. One might na\"ively expect
these conditions to be symmetric versions of the forward ones but,
with the exception of certain special cases (tent maps of {\em
  irrational} or \emph{rational endpoint} type), this is not the
case. The essential content of the forward conditions is that every
forward sequence must be less than or equal to the kneading sequence
of the tent map~$f$, in the unimodal order. For the backward
conditions, the kneading sequence is replaced by two sequences, so
that backward sequences are bounded by a stepped line. Moreover,
these two sequences mode-lock on intervals in parameter space --- what
changes as the parameter varies within such an interval is the
location of the step between the two sequences.

In Section~\ref{sec:forward-admissibility} we review the forward
admissibility conditions. This theory is well established, but we make
some minor modifications which enable us to give admissibility
conditions which are strictly necessary and sufficient
(Lemmas~\ref{lem:admissible} and~\ref{lem:admiss-invlim-forward}),
which seem not to have appeared before. The basis of the backward
admissibility conditions is the stratification of the space of
unimodal maps by \emph{height}, a number in~$[0,1/2]$ which is
associated to each unimodal map~\cite{HS}. This theory is reviewed in
Section~\ref{sec:height}, before the main results are stated and
proved. Theorem~\ref{thm:symmetric} gives necessary and sufficient
backward admissibility conditions in the symmetric case;
Theorem~\ref{thm:non-symmetric} is the analogous result in the
non-symmetric case; and Theorem~\ref{thm:mode-lock} provides a
striking illustration of the asymmetry of forward and backward
itineraries: the maximum backward itinerary which can
be realised by a tent map mode locks on intervals of kneading
sequences.

\section{Forward admissibility}
\label{sec:forward-admissibility}

\subsection{Basic definitions}
Throughout the paper, $I=[a,b]$
is a compact interval and $f\colon I\to I$ is a tent map of slope
$\lambda\in(\sqrt{2},2)$: that is, there
is some $c\in(a,b)$ such that $f$ has constant
slope~$\lambda$ on $[a,c]$ and constant slope $-\lambda$ on
$[c,b]$. Moreover, we assume that~$I$ is the dynamical interval (or
core) of~$f$, so that $f(c)=b$ and $f(b)=a$.

Let $\{0,1\}^\N$ and $\{0,1\}^\Z$ denote the spaces of semi-infinite and
bi-infinite sequences over $\{0,1\}$, with their natural product topologies. We
denote elements of the former with lower-case letters, and of the latter with
upper-case letters. We write $\sigma\colon\{0,1\}^\N\to\{0,1\}^\N$ and
$\sigma\colon\{0,1\}^\Z\to\{0,1\}^\Z$ for the corresponding shift maps. If
$S\in\{0,1\}^\Z$, we denote by $\fd{S}$ and $\bd{S}$ the elements of
$\{0,1\}^\N$ defined by $\fd{S}_r=S_r$ and $\bd{S}_r=S_{-1-r}$ for $r\ge 0$:
therefore $\fd{\sigma^r(S)}=S_rS_{r+1}\ldots$ and
$\bd{\sigma^r(S)}=S_{r-1}S_{r-2}\ldots$ for each~$r\in\Z$. We say that $S$ {\em
does not end $0^\infty$} (respectively \emph{does not start $0^\infty$}) if
infinitely many of the entries of $\fd{S}$ (respectively~$\bd{S}$) are~$1$.

If $n\ge 1$ then a \emph{word of length~$n$} is an element of
$\{0,1\}^n$. We say that a word~$W$ is \emph{even} (respectively {\em
  odd}) if it contains an even (respectively odd) number of $1$s. If
$s\in\{0,1\}^\N$ and~$W$ is a word of length~$n$, then we write
$s=W\ldots$ to mean that $s_i=W_i$ for $0\le i\le n-1$.

We denote by $\preceq$ the \emph{unimodal order} on~$\{0,1\}^\N$ (also known as
the \emph{parity lexicographical order}), which is defined as follows: if $s$
and $t$ are distinct elements of $\{0,1\}^\N$, then $s\prec t$ if and only if
the word $s_0\,\ldots\, s_r$ is even, where $r\ge 0$ is least with
$s_r\not=t_r$. An element $s$ of $\{0,1\}^\N$ is said to be
\emph{shift-maximal} if $\sigma^r(s)\preceq s$ for all~$r\ge 0$.

There are several different approaches to assigning itineraries in
$\{0,1\}^\N$ to points of~$I$ under the action of~$f$. Those
differences which are not cosmetic are concerned with the
straightforward but vexed question of how to code the critical
point~$c$, and therefore affect the itineraries of only countably
many points. One may introduce a third symbol~$C$; make an arbitrary
choice of $0$ or $1$ as the code of the critical point; allow either
of these symbols, leading to multiple itineraries for certain points
-- an approach whose ramifications are compounded when the critical
point is periodic; or take limits \`a la
Milnor-Thurston~\cite{MT}. The approach which we adopt here is to
code~$c$ with a choice of $0$ or $1$ which depends on~$f$ in the case
where~$c$ is periodic; and to allow either code for~$c$ if~$c$ is not
periodic. This convention, as well as being ideal for our results, has
the added benefit -- quite independent of the main results of the
paper -- of leading to admissibility conditions for itineraries which
are strictly necessary and sufficient (Lemma~\ref{lem:admissible}), at
least in the case of tent maps or other unimodal maps which admit no
homtervals (i.e.\ for which distinct points have distinct itineraries).

Suppose first that~$c$ is a periodic point of~$f$, of period~$n$, and
define $\ve(f)=0$ (respectively $\ve(f)=1$) if an even (respectively
odd) number of the points $\{f^r(c)\,:\,1\le r<n\}$ lie in
$(c,b]$. Then define the \emph{itinerary} $j(x)\in\{0,1\}^\N$ of $x\in
  I$ by
\[
j(x)_r = 
\begin{cases}
 0 & \text{ if }f^r(x)\in [a,c),\\
 1 & \text{ if }f^r(x)\in (c,b],\\
 \ve(f) & \text{ if }f^r(x)=c
\end{cases}
\qquad\text{ for each }r\in\N.
\]

Define the \emph{kneading sequence} $\kappa(f)\in\{0,1\}^\N$ of $f$ by
$\kappa(f)=j(b)$. By construction, $\kappa(f)=(W\ve(f))^\infty$, where~$W\ve(f)$ is
an even word of length~$n$.

\medskip\medskip

In the case where~$c$ is not a periodic point of~$f$, we say that
$s\in\{0,1\}^\N$ \emph{is an itinerary of $x\in I$} if $f^r(x)\in[a,c]$
whenever $s_r=0$, and $f^r(x)\in[c,b]$ whenever $s_r=1$. Therefore
each $x\in I$ has a unique itinerary unless $c\in\orb(x,f)=
\{f^r(x)\,:\,r\ge 0\}$, in which
case it has exactly two itineraries. 

Define the kneading sequence $\kappa(f)\in\{0,1\}^\N$ of~$f$ to be the
itinerary of~$b$ (which is unique since $b=f(c)$ and $c$ is not
periodic). Therefore if $f^r(x)=c$ for some $r\ge 0$, then the two
itineraries of $x$ are $s_0\,\ldots\,s_{r-1}\opti \kappa(f)$ for some
$s_0, \ldots, s_{r-1}\in\{0,1\}$.

\begin{remark}
\label{rmk:order}
If~$s$ is an itinerary for~$x\in I$, then $\sigma^r(s)$ is an
itinerary for $f^r(x)$ for each $r\ge 0$, regardless of whether or
not~$c$ is a periodic point. It is standard (see for
example~\cite{CE,Devaney}) that the unimodal order on itineraries
reflects the usual order on the interval~$I$.  Since~$f$ is uniformly
expanding on each of its two branches, distinct points~$x,y\in I$
cannot share a common itinerary. If $s$ and $t$ are itineraries of $x$
and $y$, we therefore have that $x<y\implies s\prec t$; while if
$s\prec t$, then either $x<y$, or $s$ and~$t$ are the two itineraries
of $x=y$ in the case where~$c$ is not periodic.
\end{remark}

Let \[j_f = \{s\in\{0,1\}^\N\,:\,s \text{ is an itinerary of some
}x\in I\}.\] An element of~$j_f$ is said to be \emph{admissible
  (for~$f$)}.

\medskip\medskip

The \emph{inverse limit}~$\hI$ of $f\colon I\to I$ is defined by 
\[\hI = \{\bx\in I^\Z\,:\, f(x_r)=x_{r+1}\text{ for all
}r\in\Z\},\] topologized as a subspace of the product $I^\Z$. This
definition differs from the standard one, in which only indices $r\le
0$ are considered, but is homeomorphic to it, since $x_0$ determines
$x_r$ for all $r>0$, and is more convenient for our purposes. Let
$\hf\colon\hI\to\hI$ be the shift map defined by $\hf(\bx)_r=x_{r+1}$
for all~$r$, a homeomorphism which is called the \emph{natural
  extension} of~$f$. The projection $\pi_0\colon\hI\to I$ defined by
$\bx\mapsto x_0$ is a semi-conjugacy from $\hf$ to $f$.

We define itineraries of elements of~$\hI$, lying in $\{0,1\}^\Z$,
in the same way as itineraries of points of~$I$
under~$f$: they provide symbolic representations of the points of~$\hI$
which are not directly related to the dynamics of $\hf$.
Thus if~$c$ is a periodic point of~$f$, then each~$\bx\in\hI$ has
a unique itinerary $J(\bx)$ defined by
\[
J(\bx)_r = 
\begin{cases}
0 & \text{ if }x_r \in [a,c),\\
1 & \text{ if }x_r \in (c,b],\\
\ve(f) & \text{ if }x_r = c
\end{cases}
\qquad\text{ for each }r\in\Z.
\]
On the other hand, if~$c$ is not a periodic point of~$f$, then we say
that $S\in\{0,1\}^\Z$ is an itinerary of~$\bx\in\hI$ if $x_r\in[a,c]$
whenever $S_r=0$, and $x_r\in[c,b]$ whenever $S_r=1$. Therefore $\bx$
has a unique itinerary if $x_r\not=c$ for all~$r$; and has exactly two
itineraries if $x_r=c$ for some~$r$, which are $\ldots\,
S_{r-2}\,S_{r-1}\,\opti\,\kappa(f)$. Note that if~$S$ is an itinerary
for $\bx\in\hI$ and $r\in\Z$, then $\fd{\sigma^r(S)}$ is an itinerary
for $x_r\in I$ under~$f$.

Let \[J_f = \{S\in\{0,1\}^\Z\,:\, S\text{ is an itinerary of some
}\bx\in\hI\}.\] An element of $J_f$ is said to be \emph{admissible (for
  $\hI$)}. If~$S$ is admissible, then it is the itinerary of only
one~$\bx\in\hI$, since each~$x_r$ is determined by its itinerary
$\fd{\sigma^r(S)}$. The map $g\colon J_f\to\hI$ which sends each
itinerary to the corresponding element of~$\hI$ is a semiconjugacy (at
most two-to-one) between the subshift $\sigma\colon J_f\to J_f$ and
the natural extension $\hf\colon\hI\to\hI$ of~$f$. For this reason we
refer to $\sigma\colon J_f\to J_f$ as the \emph{symbolic natural
  extension} of~$f$.

\begin{remark}
\label{rmk:range}
The condition that~$\lambda>\sqrt{2}$ is equivalent to the tent map
$f$'s not being renormalizable; which is equivalent in turn to the
condition $\kappa(f)\succ 101^\infty$.

The condition that~$\lambda<2$ is equivalent to $\kappa(f)\prec
10^\infty$. We exclude the case~$\lambda=2$ to avoid having to treat it
separately in lemma and theorem statements: since every element of
$\{0,1\}^\N$ (respectively~$\{0,1\}^\Z$) is admissible for~$f$
(respectively~$\hI$) when $\lambda=2$, there is no loss in so doing.
\end{remark}

\subsection{Admissibility conditions}

The following result, which gives conditions under which an element of
$\{0,1\}^\N$ is admissible for~$f$, is well known. We nevertheless
provide a proof (following those of~\cite{CE} and~\cite{Devaney}),
since it is a key result in the paper and our definition of
itineraries is slightly non-standard.

\begin{lem}[Admissibility conditions for~$f$]
\label{lem:admissible}
Write $\kappa(f)=\kappa$.  Let $s\in\{0,1\}^\N$. Then $s\in j_f$ if
and only if the following three conditions hold:
\begin{enumerate}[(a)]
\item $\sigma^r(s)\preceq\kappa$ for all $r\ge 0$;
\item $\sigma(\kappa) \preceq s$; and
\item if $c$ is periodic and $\sigma^r(s)=\kappa$ for some $r>0$, then
  $s_{r-1}=\ve(f)$. 
\end{enumerate}
\end{lem}

\begin{proof}
Let~$s$ be an itinerary of $x\in I$. Since $a$ and~$b$ have unique
itineraries $\sigma(\kappa)$ and $\kappa$, $s$ must satisfy~(a)
and~(b) by Remark~\ref{rmk:order}. Moreover, if $c$ is periodic and
$r>0$, then 
$\sigma^r(s)=\kappa \implies \mbox{$f^r(x)=b$}\implies f^{r-1}(x)=c\implies
s_{r-1}=\ve(f)$, so that~(c) holds too.

For the converse, let $s\in\{0,1\}^\N$ satisfy~(a), (b), and (c). We
shall show that $s$ is an itinerary of some~$x\in I$. We can suppose
that $\sigma(\kappa)\prec s\prec \kappa$, since otherwise~$s$ is the
itinerary of either~$a$ or~$b$.

Suppose first that~$c$ is not periodic. Define
\begin{eqnarray*}
L &=& \{x\in I\,:\, \text{ all itineraries $t$ of $x$ have }t\prec
s\},\\
R &=& \{x\in I\,:\, \text{ all itineraries $t$ of $x$ have }t\succ
s\}.
\end{eqnarray*}
Then $a\in L$ and $b\in R$. We shall show that $L$ and $R$ are open
in~$I$, so that there is some $x\not\in L\cup R$. It is impossible
for such a point~$x$ to have one itinerary smaller than~$s$ and one larger
than~$s$, since if~$s$ lies strictly between $t_0\,\ldots\,t_{r-1}0
\kappa$ and $t_0\,\ldots\,t_{r-1}1\kappa$ then
$\sigma^{r+1}(s)\succ\kappa$, contradicting~(a). Therefore $s$ is an
itinerary of~$x$, as required.

Let $x\in L$. We need to show that if $y>x$ is sufficiently close to~$x$, then
any itinerary of~$y$ is smaller than~$s$. If $c\not\in\orb(x,f)$ then this is
obvious. If $c=f^r(x)$ for some (unique)~$r\ge 0$ then~$x$ has itineraries
$t_0\,\ldots\,t_{r-1}\opti\kappa$. We can suppose that $s_i=t_i$ for $0\le i\le
r-1$, since otherwise the result is obvious. If $t_0\,\ldots\,t_{r-1}$ is even
then, since both of the itineraries of~$x$ are smaller than~$s$, we have
$s=t_0\,\ldots\,t_{r-1}1u$ for some $u\in\{0,1\}^\N$ with $u\prec\kappa$.
Let~$j\ge 0$ be least with $u_j\not=\kappa_j$. Pick~$z>x$ sufficiently close
to~$x$ that if $y\in(x,z)$ then $f^i(y)\not=c$ for $0\le i\le r+j+2$. Then any
$y\in(x,z)$ has all itineraries of the form $t_0\,\ldots\,t_{r-1}1
\kappa_0\ldots \kappa_j\ldots$, and the result follows. The argument is
analogous if $t_0\,\ldots\,t_{r-1}$ is odd.

The proof that $R$ is open in~$I$ is similar.

\medskip\medskip

Now suppose that~$c$ is periodic of period~$n$. Write
  $\ve=\ve(f)$ and $\bve = 1-\ve(f)$. Let $W\in\{0,1\}^{n-1}$ be such
  that $\kappa = (W\ve)^\infty$. Define
\begin{eqnarray*}
L &=& \{x\in I\,:\, j(x) \prec s\},\\
R &=& \{x\in I\,:\, j(x) \succ s\}.
\end{eqnarray*}
Since $a\in L$ and $b\in R$, it suffices to show that $L$ and $R$ are
open in~$I$.

Let $x\in L$. We need to show that, for all $y>x$ sufficiently close
to~$x$, we have $j(y)\prec s$. If $c\not\in\orb(x,f)$ then this is
obvious, so we suppose that there is some least~$r\ge0$ with
$f^r(x)=c$. Therefore
\[
t := j(x) = t_0\,\ldots\,t_{r-1}\,\ve\,(W\ve)^\infty.
\]
We can suppose that $s_i=t_i$ for $0\le i\le r-1$, since otherwise the
result is obvious. We distinguish two cases:

\medskip\noindent\textbf{Case A: }$t_0\,\ldots\,t_{r-1}\,\ve$ is an
even word. Since $t\prec s$ and $\sigma^{r+1}(s)\preceq(W\ve)^\infty$ by~(a) we
have that $ s = t_0\,\ldots\,t_{r-1}\,\bve\,u $ for some $u\in\{0,1\}^\N$,
which satisfies $u\prec(W\ve)^\infty$ by~(c). Write $u = (W\ve)^k\,v$ for $k\ge
0$ as large as possible, so that $v\in\{0,1\}^\N$ satisfies
$v\prec(W\ve)^\infty$ and does not have $W\ve$ as an initial subword. In
particular, $v\prec w$ for any $w\in\{0,1\}^\N$ which does start with~$W\ve$.
We have $s = t_0\,\ldots\,t_{r-1}\,\bve\,(W\ve)^k\,v$.

Pick $z>x$ sufficiently close to~$x$ that if $y\in (x,z)$ then
$f^i(y)\not=c$ for $0\le i\le (k+1)n+r+1$. Then
$j(y)=t_0\,\ldots\,t_{r-1}\,\bve\,(W\ve)^{k+1}\,\ldots$ for all
$y\in(x,z)$. This is because $f^r(y)$ is less than~$c$ (respectively
greater than~$c$) if $t_0\,\ldots\,t_{r-1}$ is odd (respectively
even); and \mbox{$f^{r+1}(y)<b$}. Since
$t_0\,\ldots\,t_{r-1}\,\bve\,(W\ve)^k$ is an odd word, it follows that
$j(y)\prec s$ for all such~$y$, as required.

\medskip\noindent\textbf{Case B: }$t_0\,\ldots\,t_{r-1}\,\ve$ is an
odd word. Since $t\prec s$ we have $s=t_0\,\ldots\,t_{r-1}\,\ve\,u$ for some
$u\in\{0,1\}^\N$ with $u\prec(W\ve)^\infty$ (this last statement again since
$t\prec s$). Write $u=(W\ve)^k\,v$ for $k\ge 0$ as large as possible: then $s =
t_0\,\ldots\,t_{r-1}\,\ve\,(W\ve)^k\,v$, where $v\prec(W\ve)^\infty$ does not
have $W\ve$ as an initial subword.

As in Case~A, if $y>x$ is close enough to~$x$ then
$j(y)=t_0\,\ldots\,t_{r-1}\,\ve\,(W\ve)^{k+1}\ldots$. This is
because $f^r(y)$ is less than~$c$ (respectively greater than~$c$) if
$t_0\,\ldots\,t_{r-1}$ is odd (respectively even). Since
$t_0\,\ldots\,t_{r-1}\,\ve\,(W\ve)^k$ is an odd word, the result follows.

\medskip

The proof that~$R$ is open in~$I$ is analogous: in this case, the
argument when $t_0\,\ldots\, t_{r-1}\,\ve$ is even is similar to
case~B, while the argument when $t_0\,\ldots\, t_{r-1}\,\ve$ is odd is
similar to case~A.
\end{proof}

The following straightforward lemma enables us to convert
these conditions into admissibility conditions for~$\hI$.

\begin{lem}
\label{lem:admiss-equiv}
Let $S\in\{0,1\}^\Z$. Then $S\in J_f$ if and only if $\fd{\sigma^r(S)}\in
j_f$ for all $r\in\Z$.
\end{lem}
\begin{proof}
If $S\in J_f$, then $S$ is an itinerary of some $\bx\in\hI$. For each
$r\in\Z$, $\fd{\sigma^r(S)}$ is an itinerary of $x_r\in I$, and hence
lies in $j_f$.

Conversely, suppose that $\fd{\sigma^r(S)}\in j_f$ for all
$r\in\Z$. Then for each~$r\in\Z$ there is a unique $x_r\in I$ which
has~$\fd{\sigma^r(S)}$ as an itinerary. Now if $s\in\{0,1\}^\N$ is an
itinerary for~$x\in I$, then $\sigma(s)$ is an itinerary for~$f(x)$:
hence $f(x_r)=x_{r+1}$ for each~$r$. Therefore $\bx=(x_r)$ is an
element of $\hI$ with itinerary $S$.
\end{proof}

\begin{lem}[Forward admissibility conditions for $\hI$]
\label{lem:admiss-invlim-forward}
Write $\kappa(f)=\kappa$. Let $S\in\{0,1\}^\Z$. Then $S\in J_f$ if and
only if the following three conditions hold:
\begin{enumerate}[(A)]
\item $\fd{\sigma^r(S)}\preceq\kappa$ for all $r\in\Z$;
\item $S$ does not start $0^\infty$; and
\item If $c$ is periodic and $\fd{\sigma^r(S)}=\kappa$ for some
  $r\in\Z$, then $S_{r-1}=\ve(f)$.
\end{enumerate}
\end{lem}
\begin{proof}
Suppose that conditions (A), (B), and (C) hold. Let~$r\in\Z$. By
Lemma~\ref{lem:admiss-equiv}, we need to show that $\fd{\sigma^r(S)}$
satisfies the conditions of Lemma~\ref{lem:admissible}. Conditions~(a)
and~(c) of Lemma~\ref{lem:admissible} are immediate from~(A)
and~(C). For~(b), suppose for a contradiction that
$\fd{\sigma^r(S)}\prec\sigma(\kappa)$. By~(B), there is some
greatest~$i<r$ with $S_i=1$. Then
$\fd{\sigma^i(S)}=10^{r-i-1}\fd{\sigma^r(S)}\succ\kappa$, which
contradicts~(A).

For the converse, suppose that $\fd{\sigma^r(S)}$ satisfies the
conditions of Lemma~\ref{lem:admissible} for each~$r\in\Z$. We need to
show that~$S$ satisfies conditions (A), (B), and
(C). Conditions~(A) and~(C) are immediate from~(a) and~(c) of
Lemma~\ref{lem:admissible}. For~(B), suppose for a contradiction that
there is some~$R\in\Z$ such that $S_r=0$ for all $r\le R$. Since
$\kappa\prec 10^\infty$, there is some $k>0$ with
$\kappa=10^k1\ldots$. Then $\fd{\sigma^{R-k}(S)} = 0^{k+1}\ldots \prec
\sigma(\kappa)$, contradicting~(b) of Lemma~\ref{lem:admissible}.  
\end{proof}

\section{Backward admissibility}
\label{sec:backward-admissibility}
\subsection{Height}
\label{sec:height}
In order to establish backward admissibility conditions we will use
the \emph{height function} $q\colon\{0,1\}^\N\to[0,1/2]$ introduced
in~\cite{HS}. Here we recall the definition of this function, and
state those of its properties which we will use.

\begin{conv}
All rationals~$m/n$ will be assumed to be written in lowest terms.
\end{conv}

Let $q\in(0,1/2]$. We associate to~$q$ a sequence $(k_i(q))_{i\ge 1}$
  in~$\N$ as follows. Let $L_q$ be the straight line $y=qx$ in $\R^2$. For each~$i\ge
  1$, define $k_i(q)$ to be two less than the number of vertical
  lines $x=\text{integer}$ which $L_q$ intersects for $y\in[i-1, i]$.

If $q=m/n$ is rational, then define the word~$c_q\in\{0,1\}^{n+1}$ by
\[
c_q = 10^{k_1(q)}110^{k_2(q)}11 \ldots 110^{k_m(q)}1.
\]
On the other hand, if $q$ is irrational,
then let $c_q=10^{k_1(q)}110^{k_2(q)}11 \ldots \in\{0,1\}^\N$. (These
sequences $c_q$ are closely related to Sturmian sequences of
slope~$q$, which can also be defined as cutting sequences of~$L_q$,
see for example~\cite{Arnoux}.)

\begin{example}
Figure~\ref{fig:cq} shows the line $L_{5/17}$ for $x\in[0,17]$. The
numbers of intersections with vertical coordinate lines for
$y\in[i-1,i]$ are $4$, $3$, $4$, $3$, and $4$ for $i=1$, $i=2$, $i=3$,
$i=4$, and $i=5$ respectively. Hence
$k_1(5/17)=k_3(5/17)=k_5(5/17)=2$, while
$k_2(5/17)=k_4(5/17)=1$. Therefore
$c_{5/17}=100110110011011001$, a word of length~18.

More generally, if $q=m/n$ then the word~$c_q$ is evidently
palindromic, and contains $n-2m+1$ zeroes divided `as even-handedly as
possible' into $m$ (possibly empty) subwords, separated by~$11$. For
example, for each~$n\ge 2$ we have $c_{1/n} = 10^{n-1}1$;
$c_{2/(2n+1)} = 10^{n-1}110^{n-1}1$;
$c_{3/(3n+1)}=10^{n-1}110^{n-2}110^{n-1}1$; and
$c_{3/(3n+2)}=10^{n-1}110^{n-1}110^{n-1}1$.
\end{example}

\begin{figure}[htbp]
\begin{center}
\includegraphics[width=0.6\textwidth]{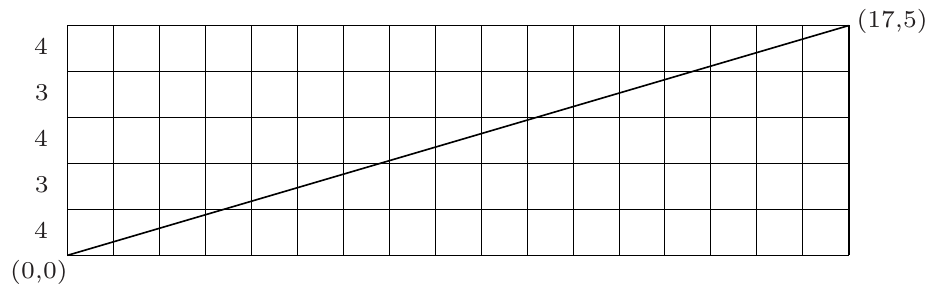}
\end{center}
\caption{$c_{5/17}=100110110011011001$}
\label{fig:cq}
\end{figure}

The following statement, which is Lemma~2.7 of~\cite{HS}, is essential
for the definition of height.

\begin{lem}
\label{lem:cq1-maximal}
 The function $(0,1/2]\cap\Q\to\{0,1\}^\N$
    defined by $q\mapsto (c_q0)^\infty$ is strictly decreasing with
    respect to the unimodal order on~$\{0,1\}^\N$. \qed
\end{lem}

We now define the \emph{height} $q(s)\in[0,1/2]$ of $s\in\{0,1\}^\N$ by
\[
q(s) = \inf\left(\{q\in(0,1/2]\cap\Q\,:\, (c_q0)^\infty \prec s\}\cup\{1/2\}\right).
\]

By Lemma~\ref{lem:cq1-maximal}, the height function
$q\colon\{0,1\}^\N\to[0,1/2]$ is decreasing with respect to the unimodal
order on~$\{0,1\}^\N$ and the usual order on~$[0,1/2]$.

In order to state the properties of height which we require, we need
some additional notation. For each rational~$q=m/n\in(0,1/2)$, let
$w_q\in\{0,1\}^{n-1}$ be defined by $(w_q)_i = (c_q)_i$ for $0\le i\le
n-2$; and let $\hw_q$ be the reverse of $w_q$, so that
$(\hw_q)_i=(w_q)_{n-2-i}$ for $0\le i\le n-2$. We therefore have
$c_q=w_q01$ for each~$q$. Since $c_q$ is an even word, $w_q$ and
$\hw_q$ are odd words.

Define $\lhe(q)$ and $\rhe(q)$ to be the shift-maximal elements of
$\{0,1\}^\N$ given by
\[\lhe(q)=(w_q1)^\infty \qquad\text{ and } \qquad \rhe(q)=c_q(1\hw_q)^\infty.\]

The first three statements of the following lemma characterize those elements
of $\{0,1\}^\N$ which have given height. Most significant, from the point of
view of this paper, is that irrational heights are realised by single elements
of $\{0,1\}^\N$, while rational heights~$q$ other than~0 are realised on
intervals in $\{0,1\}^\N$ with left and right hand endpoints $\lhe(q)$ and
$\rhe(q)$. Note (cf.\ Remark~\ref{rmk:range}) that, by~(a), the condition
$\lambda\in(\sqrt2,2)$ is equivalent to $q(\kappa(f))\in(0,1/2)$.

\begin{lem}[Properties of height] \mbox{}
\label{lem:height}
\begin{enumerate}[(a)]
\item Let $s\in\{0,1\}^\N$. Then $q(s)=0$ if and only if
  $s=10^\infty$; and $q(s)=1/2$ if and only if $s\preceq 101^\infty$.
\item Let $s\in\{0,1\}^\N$ and $q\in(0,1/2)$ be rational. Then
  $q(s)=q$ if and only if $\lhe(q)\preceq s\preceq \rhe(q)$.
\item Let $s\in\{0,1\}^\N$ and $q\in(0,1/2)$ be irrational. Then
  $q(s)=q$ if and only if $s=c_q$.
\item Let $q=m/n\in(0,1/2)$, and $1\le r\le m$. Then the word
  $10^{k_r(q)+1} 110^{k_{r+1}(q)}11\ldots 11 0^{k_m(q)}1$ disagrees
  with $c_q$ within the shorter of their lengths, and is greater than
  it in the unimodal order.
\item Let $q\in(0,1/2)$ be rational, and let
  $s=c_q\ldots\in\{0,1\}^\N$. Then $q(s)\le q$.
\item Let $s\in\{0,1\}^\N$ with $q=q(s)\in(0,1/2)$ rational. Then
  either $s=\lhe(q)$, or there is some $k\ge 0$ and $t\in\{0,1\}^\N$
  such that $s=(w_q1)^k\,t$, and either $t=c_q\ldots$ or $q(t)<q$.
\item Let $\kappa$ be the kneading sequence of a tent map, with
  $q=q(\kappa)$ rational. Then either $\kappa=\lhe(q)$, or $\kappa =
  c_q\ldots$.
\end{enumerate}
\end{lem}
\begin{proof}
For (a), the characterization of height~$0$ is immediate from the
definition of height and the fact that $c_{1/n}=10^{n-1}1$; and the
characterization of height~$1/2$ is Lemma~3.3 of~\cite{HS}. (b) is
Lemma~3.4 of~\cite{HS}, and~(c) follows from the definition of height
and that $L_q$ does not pass through any integer lattice points other
than~$(0,0)$. (d) is Lemma~63 of~\cite{stars}. For~(e), we need only
observe that if $s=c_q\ldots$ then $s\succ \lhe(q)$, and use~(b).

For~(f), if $s\not=\lhe(q)$ then $s\succ\lhe(q)$. Let $k\ge 0$ be
greatest such that $s$ starts with the word $(w_q1)^k$: then
$s=(w_q1)^k\,t$, where $t\succ\lhe(q)$ does not start with
$w_q1$. Since $t\succ\lhe(q)$ we have $q(t)\le q$. If $q(t)=q$ then,
by~(b) (and recalling that $c_q=w_q01$) we have $(w_q1)^\infty \prec t
\preceq w_q01(1\hw_q)^\infty$. Therefore, since~$t$ does not start
with $w_q1$, it must start with $w_q0$; moreover, since $w_q0$ is an
odd word, $t$ must start with $w_q01=c_q$, or it would be greater
than~$\rhe(q)$. This also proves~(g), using the observation that
if~$s=\kappa$ is a kneading sequence then, by
Lemma~\ref{lem:admissible}, we must have $k=0$, since otherwise we
would have $\sigma^{nk}(\kappa)\succ\kappa$.
\end{proof}

\begin{example}
Let $q=1/3$, so that $c_q=1001$, $w_q=10$, and $\hw_q=01$. Then
$\lhe(q)=(101)^\infty$ and
$\rhe(q)=1001(101)^\infty=10(011)^\infty$. Therefore $q(s)=1/3$ if and
only if $(101)^\infty \preceq s\preceq 10(011)^\infty$.
\end{example}

We will say that $f$ is of \emph{irrational type} if $q(\kappa(f))$ is
irrational; that it is of \emph{rational (left hand or right hand) endpoint
type} if $\kappa(f)$ is equal to $\lhe(q)$ or $\rhe(q)$ respectively for some
rational~$q$; and that it is of \emph{rational interior type} otherwise.

The following result is the essential fact which makes it possible to
relate heights of forward sequences to heights of backward
sequences. The real content of the lemma is the final sentence --- the
infimal height forward is equal to the infimal height backward, for
any bi-infinite sequence~$S$ which does not start or end $0^\infty$.


\begin{lem}
\label{lem:technical}
Let $S\in\{0,1\}^\Z$.
\begin{enumerate}[(a)]
\item If $S$ does not end $0^\infty$ then $\inf_{r\in\Z}\,
  q(\bd{\sigma^r(S)}) \le \inf_{r\in\Z}\,
  q(\fd{\sigma^r(S)})$.
\item If $S$ does not start $0^\infty$ then $\inf_{r\in\Z}\,
  q(\fd{\sigma^r(S)}) \le \inf_{r\in\Z}\, q(\bd{\sigma^r(S)})$.
\end{enumerate}
In particular, if $S$ neither starts nor ends $0^\infty$, then $\inf_{r\in\Z}\,
  q(\fd{\sigma^r(S)}) = \inf_{r\in\Z}\, q(\bd{\sigma^r(S)})$.
\end{lem}
\begin{proof}
To prove~(a), suppose for a contradiction that~$S$ does not end $0^\infty$ and
  that $\inf_{r\in\Z}\, q(\fd{\sigma^r(S)}) < \inf_{r\in\Z}\,
  q(\bd{\sigma^r(S)})$. Let~$q=m/n$ be a rational with $\inf_{r\in\Z}\,
  q(\fd{\sigma^r(S)}) < q < \inf_{r\in\Z}\, q(\bd{\sigma^r(S)})$. Then there is
  some $k\in\Z$ with $q(\fd{\sigma^k(S)}) < q$: replacing~$S$ with one of its
  shifts, we can assume without loss of generality that~$k=0$, so that
  $q(\fd{S})<q$. We will show that there is some $r\in\Z$ with
  $q(\bd{\sigma^r(S)})\le q$, which will be the required contradiction to $q <
  \inf_{r\in\Z}\, q(\bd{\sigma^r(S)})$.

Since $q(\fd{S})<q$, Lemma~\ref{lem:height}(b) gives that
$\fd{S}\succ \rhe(q)=
c_q(1\hw_q)^\infty$. If $\fd{S}=c_q\ldots$, then, since $c_q$ is
palindromic, we have
$\bd{\sigma^{n+1}(S)}=c_q\ldots$, and hence
$q(\bd{\sigma^{n+1}(S)})\le q$ by Lemma~\ref{lem:height}(e). 
We therefore suppose that $c_q$ is not an initial subword of $\fd{S}$,
so that there is some $i$ with $1\le i\le m$ and some $\ell\ge 1$ with
$\fd{S} = 10^{k_1}110^{k_2}11 \ldots 11 0^{k_{i-1}} 11 0^{k_i + \ell}
1\ldots$ (we write $k_j=k_j(q)$ and use the fact that $S$ does not end
$0^\infty$ to get the final~$1$). Writing~$r$ for the length of this
initial subword of~$\fd{S}$, we have
\begin{eqnarray*}
\bd{\sigma^r(S)} &=& 10^{k_i+\ell}110^{k_{i-1}}11 \ldots 11
0^{k_2}110^{k_1}1\ldots\\ &\succeq& 10^{k_i+1}110^{k_{i-1}}11 \ldots
11 0^{k_2}110^{k_1}1\ldots \\ &=& 10^{k_{m+1-i}+1} 11 0^{k_{m+2-i}} 11
\ldots 11 0^{k_{m-1}} 11 0^{k_m}1\ldots \quad \text{(since $c_q$ is
  palindromic)}\\ &\succ& (c_q0)^\infty
\qquad\qquad\qquad\qquad\qquad\qquad\qquad\qquad\qquad\text{(by
  Lemma~\ref{lem:height}(d)),}
\end{eqnarray*}
so that $q(\bd{\sigma^r(S)})\le q$ by the definition of height, as
required.

\medskip

Statement~(b) follows by applying~(a) to the reverse of~$S$.
\end{proof}

\begin{remark}
It is possible for one of the infima to be a minimum, and the other
not to be attained. Consider, for example, the sequence~$S$ with
$\fd{S}=(101)^\infty = \lhe(1/3)$, and $\bd{S}=1^\infty$. Then
$q(\fd{S})=1/3$, but $q(\bd{\sigma^r(S)})>1/3$ for all $r\in\Z$.
\end{remark}

\subsection{Backward admissibility conditions}

In this section we will state and prove `backward' admissibility
conditions: that is, admissibility conditions which are expressed in
terms of $\bd{\sigma^r(S)}$ rather than $\fd{\sigma^r(S)}$. We do this
first in the symmetric case (Theorem~\ref{thm:symmetric}), where they
are analogous to the `forward' conditions of
Lemma~\ref{lem:admiss-invlim-forward}; and then in the non-symmetric
case (Theorem~\ref{thm:non-symmetric}), where they take a quite
different form. We start with a lemma which will be the main part of
the proof of necessity for both theorems.

\newpage 

\begin{lem}
\label{lem:upper-bounds}
Write $\kappa=\kappa(f)$ and $q=q(\kappa)$. Let $S\in J_f$.
\begin{enumerate}[(a)]
\item If $q$ is irrational, then $\bd{S}\preceq \kappa$.
\item If $q=m/n$ is rational, then $\bd{S}\preceq\rhe(q)$. Moreover,
  if either $\kappa=\lhe(q)$ or $\fd{S}\succ \sigma^{n+1}(\kappa)$,
  then $\bd{S}\preceq\lhe(q)$.
\end{enumerate}
\end{lem}
\begin{proof}
By Lemma~\ref{lem:admiss-invlim-forward}, $S$ does not start $0^\infty$.
\begin{enumerate}[(a)]
\item Suppose for a contradiction that~$q$ is irrational and that
  $\bd{S}\succ\kappa$. Then $q(\bd{S})<q$, since, by
  Lemma~\ref{lem:height}(c), $\kappa$ is the unique element of
  $\{0,1\}^\N$ with height~$q$. By Lemma~\ref{lem:technical}(b) there
  is some~$r$ with $q(\fd{\sigma^r(S)})<q$, so that
  $\fd{\sigma^r(S)}\succ\kappa$, again by
  Lemma~\ref{lem:height}(c). This contradicts
  Lemma~\ref{lem:admiss-invlim-forward}.
\item Now let $q=m/n$ be rational. If $\bd{S}\succ\rhe(q)$ then
  $q(\bd{S})<q$, and we get a contradiction to
  Lemma~\ref{lem:admiss-invlim-forward} as in~(a). We will therefore
  suppose that $\bd{S}\preceq\rhe(q)$ in the remainder of the proof.

To prove the `moreover' statement, consider first the case where
$\kappa=\lhe(q)$, and assume for a contradiction that
$\bd{S}\succ\lhe(q)=(w_q1)^\infty$. By Lemma~\ref{lem:height}(f), we
can write $\bd{S} = (w_q1)^kt$, where $k\ge 0$, and either $q(t)<q$ or
$t=c_q\ldots$. If $q(t)<q$ then we get a contradiction as in~(a). On
the other hand, if $t=c_q\ldots$ then
$\fd{\sigma^{-(k+1)n-1}(S)}=c_q\ldots\succ\kappa$ (since $c_q$ is
palindromic), contradicting Lemma~\ref{lem:admiss-invlim-forward}.

It remains to show that if $\lhe(q)\prec \kappa\preceq\rhe(q)$ and
$\sigma^{n+1}(\kappa)\prec\fd{S}\preceq\rhe(q)$, then $\bd{S}\preceq\lhe(q)$.
Using Lemma~\ref{lem:height}(g), we write $\kappa=c_q u$, where
$u=\sigma^{n+1}(\kappa)$. Suppose for a contradiction that
\mbox{$\bd{S}\succ\lhe(q)$}. In particular, $q(\bd{S})=q$ since
$\bd{S}\preceq\rhe(q)$.

By Lemma~\ref{lem:height}(f) we have $\bd{S}=(w_q1)^kt$ for some $k\ge 0$,
where either $q(t)<q$ or $t=c_q\ldots$. If $q(t)<q$ then we get a contradiction
to Lemma~\ref{lem:admiss-invlim-forward} as in~(a), so we can assume that
$t=c_q\ldots$. Then $\fd{\sigma^{-(k+1)n-1}(S)}=c_q(1\hw_q)^k\fd{S}
\succ c_q(1\hw_q)^ku$, since $\fd{S}\succ u$ by assumption. By
Lemma~\ref{lem:admiss-invlim-forward} we therefore have $c_q(1\hw_q)^ku \prec
\kappa = c_qu$, so that $k>0$ and $(1\hw_q)^ku
\prec u$. However this inequality gives $u\succ(1\hw_q)^\infty$, so
that $\kappa=c_qu\succ c_q(1\hw_q)^\infty=\rhe(q)$, which is the
required contradiction.

\end{enumerate}
\end{proof}

\begin{thm}[Backward admissibility conditions for~$\hI$: symmetric
    case]
\label{thm:symmetric}
Suppose that~$f$ is either of irrational type, or of rational endpoint
type. Write $\kappa(f)=\kappa$. Let $S\in\{0,1\}^\Z$. Then $S\in J_f$
if and only if the following three conditions hold:
\begin{enumerate}[(a)]
\item $\bd{\sigma^r(S)}\preceq\kappa$ for all $r\in\Z$;
\item $S$ does not end $0^\infty$; and 
\item If $f$ is of left hand endpoint type and
  $\fd{\sigma^r(S)}=\kappa$ for some $r\in\Z$, then $S_{r-1}=1$.
\end{enumerate}
\end{thm}
\begin{proof}
Suppose that $S\in J_f$. Then $\sigma^r(S)\in J_f$ for all~$r\in\Z$,
and so~(a) is a consequence of Lemma~\ref{lem:upper-bounds}. If~(b)
did not hold then (since some $S_r=1$ by
Lemma~\ref{lem:admiss-invlim-forward}(B)), there would be an~$r$
with $\fd{\sigma^r(S)}=10^\infty$, contradicting
Lemma~\ref{lem:admiss-invlim-forward}(A). Condition~(c) is immediate
from Lemma~\ref{lem:admiss-invlim-forward}(C).

For the converse, suppose that $S\not\in J_f$, so that one of
conditions~(A),~(B), and~(C) of
Lemma~\ref{lem:admiss-invlim-forward} fails. We must show that one
of conditions~(a),~(b), and~(c) is false. We write $q=q(\kappa)$.

If~(A) fails, then let $r\in\Z$ with
$\fd{\sigma^r(S)}\succ\kappa$. Suppose first that
$q(\fd{\sigma^r(S)})<q$ (which must necessarily be the case if $q$ is
irrational or $q$ is rational and $\kappa=\rhe(q)$). By
Lemma~\ref{lem:technical}(a), either~(b) is false or there is some
$i$ with $q(\bd{\sigma^i(S)})<q$, so that
$\bd{\sigma^i(S)}\succ\kappa$, and~(a) is false. It remains to
consider the case where $q=m/n$, $\kappa=\lhe(q)$ and
$q(\fd{\sigma^r(S)})=q$. By
Lemma~\ref{lem:height}(f), we have
$\fd{\sigma^r(S)}=(w_q1)^kt$, where either $q(t)<q$ or $t=c_q\ldots$.
 If $q(t)<q$ then, by Lemma~\ref{lem:technical}(a),
either (a) or~(b) is false; while if $t=c_q\ldots$ then
$\bd{\sigma^{r+(k+1)n+1}(S)}=c_q\ldots \succ\kappa$, and~(a) is
false. 

If~(B) fails, then either $S_r=0$ for all~$r$, in which case~(b) is
false; or there is some~$r$ with $\bd{\sigma^r(S)}=10^\infty$, in
which case~(a) is false. Clearly if~(C) fails then~(c) is false.
\end{proof}

\begin{remark}
It is immediate from Lemma~\ref{lem:admiss-invlim-forward} and
Theorem~\ref{thm:symmetric} that if $f$ is of irrational type, or if
$\kappa(f)=\rhe(m/n)$ for some~$m/n$, then the \emph{reversing} function
$\rho\colon\{0,1\}^\Z\to\{0,1\}^\Z$ defined by $\rho(S)_r = S_{-r}$ restricts
to a homeomorphism $J_f\to J_f$, which conjugates the symbolic natural
extension $\sigma\colon J_f\to J_f$ to its inverse. On the other hand, if
$\kappa(f)=\lhe(m/n)$, then $\rho$ does not restrict to a self-homeomorphism
of~$J_f$. For example, if $\kappa(f)=\lhe(1/3)=(101)^\infty$, then the
sequences with $\bd{S}=(101)^\infty$ and \mbox{$\fd{S}=\opti 1^\infty$} are
both admissible (and hence correspond to different points of $\hI$); while the
sequence with \mbox{$\fd{S}=(101)^\infty$} and $\bd{S}=01^\infty$ is not
admissible: and, even if our conventions were changed so that it were, it would
represent the same point as the sequence with $\bd{S}=1^\infty$.
\end{remark}

\begin{thm}[Backward admissibility conditions for $\hI$: non-symmetric
    case]
\label{thm:non-symmetric}
Suppose that~$f$ is of rational interior type. Write
$\kappa(f)=\kappa$ and $q=m/n=q(\kappa)$. Let $S\in\{0,1\}^\Z$. Then $S\in
J_f$ if and only if the following four conditions hold:
\begin{enumerate}[(a)]
\item $\bd{\sigma^r(S)}\preceq \rhe(q)$ for all $r\in\Z$;
\item $\bd{\sigma^r(S)}\preceq \lhe(q)$ for all $r\in\Z$ for which
  $\fd{\sigma^r(S)}\succ\sigma^{n+1}(\kappa)$; 
\item $S$ does not end $0^\infty$; and
\item If~$c$ is periodic and $\fd{\sigma^r(S)}=\kappa$ for some
  $r\in\Z$, then $S_{r-1}=\ve(f)$.
\end{enumerate}
\end{thm}
\begin{proof}
The argument that if $S\in J_f$ then conditions (a)\,--\,(d) hold
proceeds in the same way as in the proof of
Theorem~\ref{thm:symmetric}, using
Lemmas~\ref{lem:admiss-invlim-forward} and~\ref{lem:upper-bounds}.

Suppose, then, that $S\not\in J_f$, so that one of
conditions~(A),~(B), and~(C) of
Lemma~\ref{lem:admiss-invlim-forward} fails. As in the proof of
Theorem~\ref{thm:symmetric}, if~(B) fails then either~(c) or~(a) is
false, and if~(C) fails then~(d) is false. To complete the proof, we
show that if~(A) fails then one of~(a),~(b), or~(c) is
false. We therefore assume that there is some~$r\in\Z$ with
$\fd{\sigma^r(S)}\succ\kappa$. Since~$f$ is of interior type, we have
$\kappa=c_q\ldots$ by Lemma~\ref{lem:height}(g).

If $q(\fd{\sigma^r(S)})<q$, then by Lemma~\ref{lem:technical}(a),
either~(c) is false, or there is some~$i$ with
$q(\bd{\sigma^i(S)})<q$, so that $\bd{\sigma^i(S)}\succ \rhe(q)$,
and~(a) is false.

If $q(\fd{\sigma^r(S)})=q$, then
$\fd{\sigma^r(S)}=c_q\,t$ for some $t\in\{0,1\}^\N$. Write
$\kappa=c_q\,u$, where $u\in\{0,1\}^\N$. Since
$\fd{\sigma^r(S)}\succ\kappa$ we have $\fd{\sigma^{r+n+1}(S)}=t\succ u =
\sigma^{n+1}(\kappa)$. On the other hand we have
$\bd{\sigma^{r+n+1}(S)}=c_q\,\ldots\succ\lhe(q)$: so condition~(b) is false.
\end{proof}

At this stage it is conceivable that the conditions of
Theorem~\ref{thm:non-symmetric} are in fact a symmetric version of
those of Lemma~\ref{lem:admiss-invlim-forward}, expressed in a
different way. Our final result establishes that this is not the case,
by showing that the maximum backward itinerary which can be realised
by a tent map with given kneading sequence mode locks on rational
height intervals. This contrasts with the maximum admissible forward
itinerary, which is the kneading sequence itself.

\begin{thm}[Mode-locking of maximum backward itinerary]
\label{thm:mode-lock}
Suppose that~$f$ is of rational interior type, with
$q(\kappa(f))=q\in\Q$. Then~$s=\rhe(q)$ is the greatest element
of~$\{0,1\}^\N$ with the property that there is some $S\in J_f$ with
$\bd{S}=s$.
\end{thm}

\begin{proof}
It is immediate from Theorem~\ref{thm:non-symmetric}
that $\bd{S}\preceq\rhe(q)$ for all $S\in J_f$. It is therefore only
necessary to exhibit an element~$S$ of~$J_f$ with
$\bd{\sigma^r(S)}=\rhe(q)$ for some~$r\in\Z$.

Let~$S\in\{0,1\}^\Z$ be given by $\fd{S}=(w_q0)^\infty$ and
$\bd{S}=(1\hw_q)^\infty$, so that
\[
\bd{\sigma^{n+1}(S)} = 10\hw_q(1\hw_q)^\infty = c_q(1\hw_q)^\infty =
\rhe(q),
\]
where $q=m/n$, using $c_q=w_q01 = 10\hw_q$ (as it is palindromic).
We will show that $\fd{\sigma^r(S)}\prec\kappa(f)$ for all~$r\in\Z$,
so that~$S\in J_f$ by Lemma~\ref{lem:admiss-invlim-forward}.

The sequence $(w_q0)^\infty$ is shift-maximal, since it is the saddle-node pair
of $(w_q1)^\infty = \lhe(q)$. Moreover, there do not exist shift-maximal
sequences~$s$ with $(w_q1)^\infty\prec s\prec (w_q0)^\infty$. For such a
sequence~$s$ would necessarily have initial subword $w_q$. If $s=w_q1\ldots$,
let $k\ge 1$ be greatest such that $s=(w_q1)^kt$ for some $t\in\{0,1\}^\N$.
Then $t\succ (w_q1)^\infty$, as $s\succ (w_q1)^\infty$, and since $w_q1$ is not
an initial subword of~$t$ we have $t\succ s$, contradicting the
shift-maximality of~$s$. On the other hand, if $s=w_q0t$ for some
$t\in\{0,1\}^\N$, then $t\succ(w_q0)^\infty$, as $s\prec (w_q0)^\infty$ and
$w_q0$ is odd, and so $t\succ s$, again contradicting shift-maximality.

Since $(w_q0)^\infty$ is not the kneading sequence of a tent map (its minimal
repeating word being odd) and $\kappa(f)\succ \lhe(q)$ is shift-maximal, we
have $\kappa(f)\succ (w_q0)^\infty$. Hence, for any~$r\ge 0$, we have
$\fd{\sigma^r(S)} = \sigma^r((w_q0)^\infty) \preceq (w_q0)^\infty \prec
\kappa(f)$, establishing the result in the case $r\ge 0$.

For the case~$r<0$, write $k_i=k_i(q)$ for $1\le i\le m$, so that we have
$w_q=10^{k_1}1^20^{k_2}1^2\ldots 1^20^{k_{m-1}}1^20^{k_m-1}$. Let $r<0$. If
$\fd{\sigma^r(S)}$ does not have initial subword $10$, then clearly
$\fd{\sigma^r(S)}\prec\kappa(f)$. If it does have initial subword~$10$, then
there is some $1\le i\le m$ such that
\[
  \fd{\sigma^r(S)} = 10^{k_i}1^2 0^{k_{i+1}}1^2\ldots 1^2 0^{k_m-1}1 \ldots
  \,\,\prec \,\, 10^{k_i}1^2 0^{k_{i+1}}1^2\ldots 1^2 0^{k_m-1}0\, (w_q0)^\infty,
\] which is a shift of  the shift-maximal sequence $(w_q 0)^\infty$. Therefore
$\fd{\sigma^r(S)}\prec (w_q0)^\infty \prec \kappa(f)$ as required.

\end{proof}

\bibliographystyle{amsplain}
\bibliography{itin}

\end{document}